\documentclass[11pt]{article}

\usepackage{amsmath,amssymb,amsthm,textcomp}
\usepackage{enumerate}
\usepackage{multicol}
\usepackage{amsrefs}
\usepackage [english]{babel}
\usepackage [autostyle, english = american]{csquotes}
\MakeOuterQuote{"}

\usepackage{fullpage}
\usepackage{tikz}
\usepackage{outlines}
\usepackage{enumitem}

\setenumerate[1]{label=\Roman*.}
\setenumerate[2]{label=\Alph*.}
\setenumerate[3]{label=\roman*.}
\setenumerate[4]{label=\alph*.}

\newcommand{\nline}{\rule{\linewidth}{0.5pt}}
\newcommand\tab[1][1cm]{\hspace*{#1}}

\theoremstyle{plain}
\newtheorem{theorem}{Theorem}
\newtheorem{lemma}[theorem]{Lemma}
\newtheorem{proposition}[theorem]{Proposition}
\newtheorem{corollary}[theorem]{Corollary}
\newtheorem{conjecture}[theorem]{Conjecture}

\theoremstyle{definition}
\newtheorem{definition}[theorem]{Definition}

\newcommand{\C}{\mathbb{C}}

\newcommand{\A}{\mathbb{A}}
\newcommand{\Q}{\mathbb{Q}}
\newcommand{\Z}{\mathbb{Z}}
\newcommand{\R}{\mathbb{R}}
\newcommand{\N}{\mathbb{N}}

\newcommand{\ve}{\varepsilon}
\newcommand{\vph}{\varphi}

\newcommand{\on}{\operatorname}

\renewcommand{\Re}{\on{Re}}
\renewcommand{\Im}{\on{Im}}
\newcommand{\gm}{\gamma}

\newcommand{\sg}{\sigma}
\newcommand{\om}{\omega}

\newcommand{\wt}{\widetilde}
\newcommand{\wh}{\widehat}

\newcommand{\ssm}{\smallsetminus}

\makeatletter
\renewcommand{\maketitle}{
	\begin{center}
		\vspace{2ex}
		{\huge \textbf{\@title}} \\
		\nline \\
		{\large \@author}\\
		
		\vspace{4ex}
	\end{center}
}
\makeatother

\begin{document}

\title{The Hlawka Zeta Function as a Respectable Objecct}
\author{Mike Montoro, under the Advisement of Prof Joseph Hundley\\
Submitted as Undergraduate Honors Thesis at the State University of New York at Buffalo}
\maketitle	
\begin{abstract}
	The Hlawka Zeta Function is a Dirichlet series defined geometrically which provides an integral representation of the number of lattice points contained in the dilation $tD$ for some star shaped region $D\subset \R^{2}$ and some real number $t\in \R^{+}$. We give an overview of this construction and integral representation before giving the Hlawka Zeta function as a sum of Eisenstein Series acting on $K$-finite vectors multiplied by Fourier coefficients depending on $D$. We then study the case of $D$ as an circle, ellipse, and then square to study functional equations and "fibers" of this object, and pose conjectures regarding these properties in general.
\end{abstract}

\section{Definition of the Hlawka Zeta Function}

\tab Let $r:\R/2\pi \Z\to \R_{> 0}$ be a continuous function such that the mapping $\theta \mapsto (r(\theta)\cos(\theta),r(\theta)\sin(\theta))$ is a closed curve in $\R^{2}$. The goal is to count the number of lattice points in $\Z^{2}\hookrightarrow \R^{2}$ which are also in $tD$, where $D$ is the domain $\{(x,y) \in \R^{2}:|(x,y)| \leq r(\theta)\}$, and $tD$ is that domain dilated by the real number $t$. A crude estimate would be simply the area of the domain $tD$, but this is not an exact estimate. This discrepancy is discussed in detail in \cite{Huxley1}. In an attempt to correct this, we introduce the \textit{Hlawka Zeta Function}. \\
\tab 
\begin{definition}
Let $t(m,n):=\min\{t\in \R_{> 0} :(m,n)\in tD\}$, $(m,n)\in \Z^{2}$. Then the Hlawka Zeta function is defined as 
\begin{equation}
Z_{r}(s)=\sum_{(m,n)\in \Z^{2}\ssm \{(0,0)\}}\frac{1}{t(m,n)^{2s}}.
\end{equation}
\end{definition}
For the sake of space, we will mean $(m,n)\in \Z^{2}\ssm \{(0,0)\}$ when we write $(m,n)\neq (0,0)$ as the index in a sum. Also, we will sometimes notate $Z_{r}(s)=Z(r,s)$. There is an explicit formula for the $t(m,n)$ in terms of $r(\theta)$: note that $t(m,n)$ is the unique real number such that $|(m,n)|=|t(m,n)r(\theta(m,n))|$ when $\theta(m,n)$ is the angle between the positive $x$ axis and the ray from the origin to the point $(m,n)$, so
$$t(m,n)=\frac{|(m,n)|}{r(\theta(m,n))}=\frac{\sqrt{m^{2}+n^{2}}}{r(\theta(m,n))}.$$
Thus we can rewrite: 
\begin{equation}
Z_{r}(s)=\sum_{(m,n)\neq (0,0)}\frac{r^{2s}(\theta(m,n))}{(m^{2}+n^{2})^{s}}.
\end{equation}

\tab We wish to write $\{t(m,n)\}_{m,n=1}^{\infty}$ as a sequence $\{t_{k}\}_{k=1}^{\infty}$ with an ordering such that $t_{k}<t_{k+1}$ for all $k$. This requires a fact that will also be useful later. 

\begin{proposition}
	Let $\{t(m,n)\}_{m,n=1}^{\infty}$ be as in Definition 1. This sequence has no limit point in $\R_{\geq 0}$ (with respect to the Euclidean topology), i.e. is a discrete set. 
\end{proposition}

\begin{proof}
	Suppose $T\in \R_{> 0}$ is a limit point of $\{t(m,n)\}_{m,n=1}^{\infty}$. Then there exists an infinite subset $\{t(m_{j},n_{j})\}_{j=1}^{\infty}$ with the property that $t(m_{j},n_{j}) \in (T-\frac{1}{4},T+\frac{1}{4})$ for all $j$. Let $B_{\ve}(x,y)$ be the Euclidean ball of radius $\ve$ centered at $(x,y)$. Consider the closed sets 
	$$C_{j}=\overline{r(r^{-1}(B_{\frac{1}{4}}(m_{j},n_{j})))}$$ 
	where we consider $r$ as a function $\theta\mapsto (r(\theta)\cos(\theta),r(\theta)\sin(\theta))$.  These are clearly closed sets, and moreover $C_{i}\cap C_{j}=\emptyset$ for $i\neq j$ since each ball of radius $\frac{1}{4}$ is disjoint. Then is easy to see that $\bigcup_{j=1}^{\infty}C_{j}$ is also a closed set. Since $[0,2\pi]$ is a compact set, $r([0,2\pi])$ must be compact, and so since $\bigcup_{j=1}^{\infty}C_{j}\subset r([0,2\pi])$, $\bigcup_{j=1}^{\infty}C_{j}$ is bounded. Thus $\bigcup_{j=1}^{\infty}C_{j}$ is compact. However, the open cover $\{B_{\frac{1}{2}}(m_{j},n_{j})\}_{j=1}^{\infty}$ has no subcover, let alone a finite subcover. Thus we obtain a contradiction. 
\end{proof}
 
We now invoke a lemma from set theory. 

\begin{lemma}
	Any non-empty discrete subset $S$ of $\R_{\geq 0}$ has a least element. Moreover, if $S$ is infinite there is a bijection $f:\Z_{\geq 0} \to S$ preserving the order "$<$" thus giving $S$ a sequential form with $s_{0}<s_{1}<\cdots$. 
\end{lemma}

\begin{proof}
	We first prove the first part. If $S$ is finite, then the claim is clear. If $S$ is infinite, suppose that $S$ has no least element. Choose $s \in S$, and consider $[0,s]$. There is an element $s_{1}$ in $S$ which is less than $s$, thus in $[0,s]$. The same is true for $s_{1}$ itself, i.e. there is an element $s_{2}\in S$ with $0<s_{2}<s_{1}<s$. This process must go on infinitely, so there is an infinite subset of $S$, $\{s_{n}\}_{n=1}^{\infty}$ completely contained in $[0,s]$. But $[0,s]$ is compact, so any infinite subset must have a limit point, contradicting the discreteness of $S$. \\
	\tab To prove the second part, we see from the first part that $S$ has a least element, say $s_{0}$. Let $f(0)=s_{0}$. Now $S\ssm \{s_{0}\}$ is a discrete set as well. Thus it has a least element, $s_{1}$. Let $f(1)=s_{1}$. One may continue this process ad infinitum, thus constructing the function $f$ as desired and giving $S$ the desired sequential form  $\{s_{k}\}_{k=1}^{\infty}$. 
\end{proof}
 
Finally, we have:

\begin{corollary}
	The set $\{t(m,n)\}_{m,n=1}^{\infty}$ from Definition 1 has a sequential form $\{t_{k}\}_{k=1}^{\infty}$ with $t_{k}<t_{k+1}$ for all $k$. 
\end{corollary}

\begin{proof}
	This follows immediately from the proposition and the lemma. 
\end{proof}

\tab Define $a_{k}=|\{(m,n)\in Z^{2}\ssm \{(0,0)\}:(m,n)\in \partial t_{k}D\}|$,i.e. the number of lattice points on in $\partial t_{k}D$, $\partial t_{k}D$ denoting the boundary of $t_{k}D$. Thus we can now rewrite the Hlawka Zeta Function as 
\begin{equation}
Z_{r}(s)=\sum_{k=1}^{\infty}a_{k}t_{k}^{-2s}.
\end{equation}

The reformulation of the definition as in (2) and (3) are often much more useful in studying $Z_{r}(s)$, since this looks like a Dirichlet Series. This is a key idea we will use later. \\
\tab Given that this a definition given in terms of a sum, there is an immediate question of the convergence of $Z_{r}(s)$ and thus where $Z_{r}(s)$ is well defined. We may partially address this immediately with the following lemma.
\begin{lemma}
	$Z_{r}(s)$ converges absolutely for $\Re(s)>1$.  
\end{lemma}
Note that this result is independent of $r$. 
\begin{proof}
	Since $r$ is a continuous function on the compact set $[0,2\pi]$, there is a number $M$ with $r(\theta)<M$ for all $\theta\in [0,2\pi]$. This gives $|r(\theta)^{s}|=r(\theta)^{\Re(s)}<M^{\Re(s)}=|M^{s}|$ which holds for $\Re(s)>0$. Now, 
	$$\sum_{(m,n)\neq (0,0)}\left|\frac{r^{2s}(\theta(m,n))}{(m^{2}+n^{2})^{s}}\right| < \sum_{(m,n)\neq (0,0)}\left|\frac{(M^{2})^{s}}{(m^{2}+n^{2})^{s}}\right|=|M^{2s}|\sum_{(m,n)\neq(0,0)}\frac{1}{|(m^{2}+n^{2})^{s}|}$$
	which converges for $\Re(s)>1$. 
\end{proof}
Actually, $Z_{r}(s)$ converges for $\Re(s)>1$. This will proven in a subsequent section. \\ 
\tab Now that we have the basic definitions understood, there is a series of questions we seek to answer. 
\begin{enumerate}
	\item What geometric information does the Hlawka Zeta function give us? How does it relate to counting 
	\item What are the analytic properties of the Hlawka Zeta function? Does it have an analytic continuation beyond the domain of convergence? Namely, does it have a functional equation? 
	\item The Hlawka zeta function defines a mapping sending a (continuous) function $r:\R/2\pi \Z\to \R_{>0}$ to $Z_{r}(s)$, which is a meromorphic function on at least some half plane in the complex. What are the properties of this mapping? More specifically, when does $Z_{r_{1}}(s)=Z_{r_{2}}(s)$ for $r_{1}\neq r_{2}$?
\end{enumerate}

\section{Motivation for the Definition}

\tab Suppose $a(n)$ is a sequence of real numbers. Let $$(D.a)(s)=\sum_{n=1}^{\infty}\frac{a(n)}{n^{s}}$$ be the associated Dirichlet series. Let $A(x)=\sum_{n\leq x}a(n)$, and let $A'(x)=A(x)$ everywhere but at integers $x=n$, and at those points set $A'(n)=a_{1}+\cdots a_{n-1}+\frac{1}{2}a_{n}$. The definition of the Hlawka Zeta function arises from the following well-known result: 
\begin{theorem}[Perron's Formula]
	Suppose $(D.a)(s)$ converges absolutely when $\Re(s)>c$. We have the following formula: 
	$$A'(x)=\frac{1}{2\pi i}\int_{\sg-i\infty}^{\sg+i\infty} (D.a)(s) x^{s}\frac{ds}{s}$$
	where $\sg > \max\{0,c\}$. 
\end{theorem}

Note that it is possible that the Dirichlet series could never converge, making the statement of the formula vacuous. We will prove this from another well known transformation: 
\begin{definition}
	Let $\phi:\R\to \R$ be a real valued function which is continuous except at countably many places $\{y_{n}\}_{n=1}^{\infty}$ with no limit point in the domain and at these places, $\phi(y_{i})=\lim\limits_{y\to y_{i}^{+}}\phi(y)$ or $\phi(y_{i})=\lim\limits_{y\to y_{i}^{-}}\phi(y)$, and
	$$L_{i}=\lim\limits_{y\to y_{i}^{+}}\phi(y)-\lim\limits_{y\to y_{i}^{-}}\phi(y)$$
	 is finite. Then the \textit{Mellin Transform} of this function is defined as 
	$$\Phi(s)=\int_{0}^{\infty}\phi(y)y^{s}\frac{dy}{y}.$$
	We write $\Phi(s)=(\mathcal{M}.\phi)(s)$.
\end{definition}
This is a well-defined complex function in $\sg_{1}<\Re(s)<\sg_{2}$ for some $\sg_{1},\sg_{2}\in [-\infty,\infty]$, determined by the convergence of the integral from 1 to $\infty$ and from 0 to 1, respectively. It is worth noting that when we say "limit point", we use the Euclidean topology on $\R$. \\
\tab For functions of the type in the definition, define the \textit{sign of the discontinuity at $y_{i}$}, denoted by $SGN(\phi,y_{i})$, to be $+1$ if $\phi(y_{i})=\lim\limits_{y\to y_{i}^{-}}\phi(y)$ and $-1$ if $\phi(y_{i})=\lim\limits_{y\to y_{i}^{+}}\phi(y)$. We also have this key result.
\begin{lemma}[Mellin Inversion formula]
	Let the situation be as above. Then 
	$$\phi'(y)=\frac{1}{2\pi i}\int_{\sg-i\infty}^{\sg+i\infty}\Phi(s)y^{-s}ds,$$
	for $\sg\in (\sg_{1},\sg_{2})$, where $\phi'(y)=\phi(y)$ for all $y\in\R \ssm \{\{y_{i}\}_{n=1}^{\infty}\}$ and 
	$$\phi'(y_{i})=\phi(y_{i})+\epsilon\frac{1}{2}|L_{i}|$$
	where $\epsilon=SGN(\phi,y_{i})sgn(L_{i})$, where $sgn$ is the normal signum function. We then define the Mellin inversion operator $(\mathcal{M}^{-1}.\Phi(s))=\phi(y)$ by the above formula.
\end{lemma}
Note that $\phi'(y_{i})$ is more simply stated as the value halfway between the two limits as $\phi$ approaches $y_{i}$ from the left and the right, sometimes denoted by $\phi'(y_{i})=\frac{1}{2}(\phi(y_{i}+0^{+})+\phi(y_{i}+0^{-})$ in the literature. The function $A(x)$ gives $A'(x)$ in this way. Moreover, this says that $\phi$ is recoverable from $\Phi$ except on a set of measure zero. The proof of this can be found in \cite{Tmarsh}, so we omit it. We now prove Perron's Formula using this lemma.   
\begin{proof}
	Assume the Dirichlet series $(D.a)(s)$ converges absolutely for $\Re(s)>c$. We look at the integral 
	$$\mathcal{M}.(A'(x))=\int_{0}^{\infty}\sum_{n=1}^{\lfloor x \rfloor}a(n)x^{s-1}dx.$$ 
	We first note that the inside is zero for $x<1$. Thus this is equal to 
	$$\int_{1}^{\infty} \sum_{n=1}^{\lfloor x \rfloor} a(n)x^{s-1}=\sum_{n=1}^{\infty}\int_{n}^{\infty}a(n)x^{s-1}dx,$$ 
	if the sum on the right is absolutely convergent. Evaluating that integral gives 
	\begin{equation}
	\int_{n}^{\infty}a(n)x^{s-1}dx=\lim\limits_{b\to \infty}[a(n)x^{s}/s]_{x=n}^{x=b} 
	\end{equation}
	and we see that $\lim\limits_{b\to \infty}\frac{a(n)b^{s}}{s}=0$ if $\Re(s)<0$. Then the whole sum is now: $\sum_{n=1}^{\infty}\frac{-a(n)n^{s}}{s}=\frac{-1}{s}(D.a)(-s)$. This sum converges absolutely for $-\Re(s)>c$, i.e. $\Re(s)<-c$ and $\Re(s)<0$. Thus, the two sums in (4) are equal for $\Re(s)>\max\{0,c\}$ since the sum on the left converges absolutely for those values of $s$. We then apply the Mellin inversion formula to obtain $A'(x)=\mathcal{M}^{-1}.(\frac{-1}{s}(D.a(-s)))$ for $\sg<\max\{0,c\}$, and this is now 
	$$A'(x)=\mathcal{M}^{-1}.(\frac{-1}{s}(D.a(-s)))=\frac{1}{2\pi i}\int_{\sg-i\infty}^{\sg+i\infty}\frac{-1}{s}(D.a)(-s)x^{-s}ds=\frac{1}{2\pi i}\int_{-\sg-i\infty}^{-\sg+i\infty}(D.a)(s)x^{s}\frac{ds}{s},$$ where $-\sg>\max\{0,c\}$. 
\end{proof}

This is an elegant formula. Our proof is non-standard for a very specific reason: It is virtually useless to study the Hlawka Zeta function, since usually $Z_{r}(s)$ is not a standard Dirichlet Series. We will need something slightly more general. We develop a basic set of definitions about $t$-Dirichlet series.  
\begin{definition}
	Let $a$ be a function $a:\R\to \R$, and $\{t_{n}\}_{n=1}^{\infty}$ a sequence of positive real numbers. Define $(D_{t}.a)(s)$ as 
	$$(D_{t}.a)(s)=\sum_{n=1}^{\infty}\frac{a(t_{n})}{(t_{n})^{s}},$$ the $t$-Dirichlet series attached to the pair $(a,t_{n})$
\end{definition}
Notably, this is equivalent to the so-called Generalized Dirichlet Series. We will reconcile these definitions later.\\
\tab Define $A_{t}(x)=\sum_{t_{n}\leq x}a(t_{n})$ for some $a:\R\to \R$. Note that $a\circ t : \N \to \R$ is also a sequence of real numbers, so $A_{t}(x)$ also gives the sum of some sequence, but one indexed by $t_{n}$, a more general sequence of real numbers. With this definition, we see something relatively remarkable happen.
\begin{theorem}[Perron's Formula for $t$-Dirichlet Series.]
	Suppose that the $t$-Dirichlet series attached to the pair $(a,t_{n})$ converges absolutely for $\Re(s)>c$ and that $\{t_{n}\}$ has no limit point in $\R_{\geq 0}$. We then have the following: 
	$$A_{t}'(x)=\frac{1}{2\pi i}\int_{\sg-i\infty}^{\sg+i\infty}(D_{t}.a)(s)x^{s}\frac{ds}{s}$$
	for $\sg>\max\{0,c\}$.
	Here $A_{t}(x)=A_{t}'(x)$ except at $t_{n}$ and $A_{t}'(t_{n})$ has its last term multiplied by $\frac{1}{2}$. 
\end{theorem}

\begin{proof}
	This is the reason we proved the original Perron's formula the way we did. The key point is that in fact $A_{t}(x)$ has a Mellin Transform since $t_{n}$ has no limit point. Then the above can be proven identically by simply replacing the limits of integration from $n$ to $t_{n}$. The integral still stays away from $0$, so any convergence issue there is fine. The rest is clear. 
\end{proof}

\tab Now we see the immediate connection here between these $t$-Dirichlet series and the Hlawka Zeta function: that the Hlawka Zeta function is in fact a $t$-Dirichlet series evaluated at $2s$, using the pair $(a_{k},t_{k})$ as defined in the first section. Furthermore, the key point about $\{t_{k}\}_{k=1}^{\infty}$ being discrete has already been previously addressed in the first section, so we immediately have the following corollary. 
\begin{corollary}[Point-Count Formula]
With notation as in the first section, let $A(x)$ be the sum $\sum_{t_{n} \leq x} a_{k}$. This is precisely the amount of lattice points in the region $xD$. As in Perron's formula, let $A'(x)$ be the same sum but with final term multiplied by $\frac{1}{2}$ at $A'(t_{k})$. The Hlawka Zeta function satisfies 
$$A'(x)=\frac{1}{2\pi i}\int_{\sg-i\infty}^{\sg+i\infty} \left(\sum_{n=1}^{\infty}\frac{a_{k}}{t_{k}^{s}}\right)x^{s}\frac{ds}{s}=\frac{1}{2\pi i}\int_{2\sg+i\infty}^{2\sg-i\infty}Z_{r}(s)x^{2s}\frac{ds}{s}$$  
for $\sg>1$. 
\end{corollary}
Here we cut out the intermittent function in the $a$, i.e. $a(t_{k})=a_{k}$. The Hlawka Zeta function thus literally plays an integral role in counting lattice points on in a star-shaped region expanding at a constant rate. This answers one of our main questions. 

\section{Formula as a sum over Eisenstein Series}

We now want to establish a formula for the Hlawka zeta function in terms of Eisenstein Series acting on a particular set of functions. We first define this class of function now. To standardize notation, let $GL(2,\R)^{+}$ be real 2 by 2 matrices with positive determinant. Let $f_{q,s_{1},s_{2}}:GL(2,\R)^{+}\to \C$, $q\in \Z$, $s_{1},s_{2}\in \C$ be defined in the following manner. The Iwasawa decomposition of $GL(2,\R)^{+}$ gives that an element of this group, $g$, can be written uniquely as \\
\begin{equation}
g=\begin{pmatrix}
u & 0 \\
0 & u \\
\end{pmatrix}
\begin{pmatrix}
y^{1/2} &  xy^{-1/2} \\
0 & y^{-1/2} \\
\end{pmatrix} \kappa_{\theta}, \ \kappa_{\theta}=
\begin{pmatrix}
\cos(\theta) & \sin(\theta) \\
-\sin(\theta) & \cos(\theta) \\
\end{pmatrix}
\end{equation}
for $x,y,u\in \R$ and $u,y>0$, $\theta\in [0,2\pi)$. This allows us to define 
$$f_{q,s_{1},s_{2}}\left(
\begin{pmatrix}
u & 0 \\
0 & u \\
\end{pmatrix}
\begin{pmatrix}
y^{1/2} &  xy^{-1/2} \\
0 & y^{-1/2} \\
\end{pmatrix} \kappa_{\theta} \right) :=u^{s_{1}+s_{2}}y^{(s_{1}-s_{2}+1)/2}e^{iq\theta}.$$
For our purposes, we will want to look at these functions as defined on the quotient $\Gamma^{\infty}_{+} \backslash  GL(2,\R)^{+}$, where 
$$\Gamma^{\infty}_{+}=\left\{ 
\begin{pmatrix}
1 & n \\
0 & 1 \\ 
\end{pmatrix}:n\in \Z \right\}, \ n\in \Z.$$
To establish this as reasonable, we have the following.
\begin{lemma}
	The functions $f_{q,s_{1},s_{2}}$ are equal on right cosets $\Gamma_{+}^{\infty}g$ for $g\in GL(2,\R)^{+}$, thus defining functions $\wt{f}_{q,s_{1},s_{2}}:\Gamma_{+}^{\infty}\backslash GL(2,\R)^{+}\to \C^{\times}$ uniquely on the quotient group. Furthermore if $s=(s_{1}-s_{2}+1)/2$, we fix $s_{1}+s_{2}$, and $\theta(c,d)$ is the angle between the line from the origin and $(c,d)$ and the positive $x$-axis then
	$$f_{q,s_{1},s_{2}}\left(
	\begin{pmatrix}
	a & b \\
	c & d 
	\end{pmatrix}\right)=\frac{(-i)^{q}e^{iq\theta(c,d)}}{(c^{2}+d^{2})^{s}} \ \text{for} \ 	\begin{pmatrix}
	a & b \\
	c & d 
	\end{pmatrix} \in SL(2,\Z).$$
	Thus $\wt{f}_{q}:\C\times \Gamma_{+}^{\infty}\backslash SL(2,\Z)\to \C^{\times}$ where $\wt{f}_{q}(s,\om g)=f_{q,s_{1},s_{2}}(g)$ and $\om\in \Gamma_{+}^{\infty}$, $g\in SL(2,\Z)$ is a well-defined, smooth function of $s$ as well. 
\end{lemma}
\begin{proof}
For an element $g=
\begin{pmatrix}
a & b \\
c & d \\
\end{pmatrix}\in GL(2,\R)^{+}$ an easy computation gives $y=((c/u)^{2}+(d/u)^{2})^{-1}$ in the Iwasawa decomposition. Furthermore, if $\theta(c,d)$ is the angle made between the point $(c,d)$ and the positive $x$-axis, $\theta=\theta(c,d)-\pi/2$ (this can be seen from a right triangle construction). The formula for $f_{q,s_{1},s_{2}}$ is then 
$$f_{q,s_{1},s_{2}}(g)=u^{s_{1}+s_{2}}((c/u)^{2}+(d/u)^{2})^{-(s_{1}-s_{2}+1)/2}e^{iq(\theta(c,d)-\pi/2)}.$$
With this in mind, we compute
$$\begin{pmatrix}
1 & n \\
0 & 1 \\
\end{pmatrix}
\begin{pmatrix}
a & b \\
c & d \\
\end{pmatrix}=
\begin{pmatrix}
a+nc & b+nd \\
c & d \\
\end{pmatrix}$$
so multiplying by an element of $\Gamma_{+}^{\infty}$ on the left affects neither $u$ nor the bottom row of the matrix. It follows that $f_{q,s_{1},s_{2}}$ is well defined on $\Gamma^{\infty}_{+}\backslash GL(2,\R)^{+}$.
We now restrict to the subgroup $SL(2,\Z)$ and let $s=(s_{1}-s_{2}+1)/2$ so that this becomes 
$$f_{q,s_{1},s_{2}}(g)=\frac{(-i)^{q}e^{iq\theta(c,d)}}{(c^{2}+d^{2})^{s}}.$$
 The smoothness in $s$ is then clear from the formula. 
\end{proof}
We now introduce the key Eisenstein Series in the upcoming proposition. 
\begin{definition}
The Eisenstein Series acting on $f_{q}$ evaluated at $g\in GL(2,\R)^{+}$ is 
$$(E.f_{q})(g,s_{1},s_{2})=\sum_{\gm\in \Gamma^{\infty}_{+} \backslash  SL(2,\Z)} \wt{f}_{q,s_{1},s_{2}}(\gm g).$$
\end{definition}

We see that due to the lemma if $g\in SL(2,\Z)$ we may write $$(E.f_{q})(g,s_{1},s_{2})=\sum_{\gm\in \Gamma^{\infty}_{+} \backslash  SL(2,\Z)} \wt{f}_{q}(s,\gm g)=(E.f_{q})(g,s)$$
where $s=\frac{1}{2}(s_{1}-s_{2}+1)$ as before in this case. Furthermore, we have
$$(E.f_{q})(I,s)=\sum_{\gcd(c,d)=1}\frac{(-i)^{q}e^{iq\theta(c,d)}}{(c^{2}+d^{2})^{s}}.$$
Also, we define 
$$(E^{\sharp}.f_{q})(g,s_{1},s_{2}):=\zeta(2s)(E.f)(g,s_{1},s_{2}).$$ 
One can verify through a standard calculation (noting that $\theta(m,n)=\theta(km,kn)$ for $k>0$) that 
$$(E^{\sharp}.f_{q})(I,s)=\sum_{(m,n)\neq (0,0)}\frac{(-i)^{q}e^{iq\theta(m,n)}}{(m^{2}+n^{2})^{s}}.$$
Note that this Eisenstein Series converges absolutely for $\Re(s)>1$, as the modulus of each term is the same as a classical Eisenstein series, so we refer to \cite{Bump1}. \\
\tab We now define also $\wh{r^{s}}(q)$ to be the $q$th Fourier coefficient in the Fourier Series for $r^{s}$, i.e. $r^{s}(\theta)=\sum_{q=-\infty}^{\infty}\wh{r^{s}}(q)e^{iq\theta}$ when 
$$\wh{r^{s}}(q)=\frac{1}{2\pi}\int_{0}^{2\pi}r^{s}(\theta)e^{-iq\theta}d\theta.$$
Since 
$$r^{s}(\theta)=r(\theta)^{i\Im(s)+\Re(s)}=r(\theta)^{\Re(s)}e^{i\ln(r(\theta))\Im(s)}$$
$r^{s}$ is a continuous function from $\R/2\pi\to \C$. We will need to discuss the absolute convergence of this Fourier series. We introduce a definition here to help this classification.  
\begin{definition}
	Let $A\subseteq \R^{n}$ (or $\C^{n}$) be an open set and $f:A \to \R^{m}$ (or $\C^{m}$). If there exists a constant $C$ and $0<\alpha\leq 1$ such that 
	$$|f(x)-f(y)|\leq C |x-y|^{\alpha}$$ 
	for all $x,y\in A$, then call $f$ \underline{$\alpha$-H\"{o}lder continuous}. 
\end{definition} 
This definition makes sense for any function between metric spaces. We note that $\alpha>0$ ensures that the function is continuous, and that $\alpha=1$ would give the definition for Lipschitz continuous. Thus this condition is stronger than simply continuous and weaker than Lipschitz continuous or continuously differentiable.  \\
\tab We have the following result.  

\begin{lemma}
	If $f:\R/2\pi \Z \to \C$ is $\alpha$-H\"{o}lder continuous for $\alpha>\frac{1}{2}$, then its Fourier Series converges absolutely. 
\end{lemma}
\begin{proof}
	See \cite[31-33]{katz}. 
\end{proof}

Note here that there are in fact conditions which exactly give the absolute convergence of the Fourier Series, but we opt for this condition since it is enough to check that $r$ is either differentiable or Lipschitz for this to hold. \\
\tab With this notation, we can now state the main theorem of this section. 
\begin{theorem}[Eisenstein Series Formula, preliminary version]
	Let all of the definitions be as above. Suppose $r$ is $\alpha$-H\"{o}lder continuous for $\alpha>\frac{1}{2}$. Then the Hlawka Zeta function $Z_{r}(s)$ is 
	$$Z_{r}(s)=\sum_{q=-\infty}^{\infty}i^{q}\wh{r^{2s}}(q)(E^{\sharp}.f_{q})(I,s)$$ 
	with absolute convergence of the right side happening on the half-plane $\Re(s)>1$. 
\end{theorem}

\begin{proof}
	We look to evaluate the sum in question: 
	$$\sum_{q=-\infty}^{\infty}\wh{r^{2s}}(q)(E^{\sharp}.f_{q,s})(I,s)=\sum_{q=-\infty}^{\infty}\wh{r^{2s}}(q)\sum_{(m,n)\neq (0,0)}\frac{e^{iq\theta(m,n)}}{(m^{2}+n^{2})^{s}}=\sum_{q=-\infty}^{\infty}\sum_{(m,n)\neq (0,0)}\frac{\wh{r^{2s}}(q)e^{iq\theta(m,n)}}{(m^{2}+n^{2})^{s}}.$$
	We would like to change the order of the summation: Thus we must check that the sum absolutely converges. We see 
	$$\sum_{q=-\infty}^{\infty}\sum_{(m,n)\neq (0,0)}\left|\frac{\wh{r^{2s}}(q)e^{iq\theta(m,n)}}{(m^{2}+n^{2})^{s}}\right|=\sum_{q=-\infty}^{\infty}|\wh{r^{2s}}(q)|\sum_{(m,n)\neq (0,0)}\left|\frac{1}{(m^{2}+n^{2})^{s}}\right|$$
	and we know that $\sum_{(m,n)\neq (0,0)}\frac{1}{(m^{2}+n^{2})^{s}}$ converges absolutely with $\Re(s)>1$ and $\sum_{q=-\infty}^{\infty}|\wh{r^{2s}}(q)|$ converges by the lemma. Thus we can switch the order of the summation to obtain: 
	$$\sum_{q=-\infty}^{\infty}\sum_{(m,n)\neq (0,0)}\frac{\wh{r^{2s}}(q)e^{iq\theta(m,n)}}{(m^{2}+n^{2})^{s}}=\sum_{(m,n)\neq(0,0)}\frac{1}{(m^{2}+n^{2})^{s}}\sum_{q=-\infty}^{\infty}\wh{r^{2s}}(q)e^{iq\theta(m,n)}$$
	and now since $\sum_{q=-\infty}^{\infty}\wh{r^{2s}}(q)e^{iq\theta(m,n)}=r^{2s}(\theta(m,n))$, we get that this is equal to 
	$$=\sum_{(m,n)\neq (0,0)}\frac{r^{2s}(\theta(m,n))}{(m^{2}+n^{2})^{s}}$$ 
	which is one of the definitions of $Z_{r}(s)$. 
\end{proof}

\tab This proposition tells us that the Hlawka Zeta function is a sum of Eisenstein Series with coefficients given by these $\wh{r^{2s}}(q)$. The Eisenstein Series portion is fairly well known and studied. The coefficients on the other hand are much more difficult to compute. In fact, much of the information about $Z_{r}(s)$ is sitting inside of these coefficients. Before analyzing $Z_{r}(s)$ under this context, there is an easy corollary to this proposition which will prove useful later. \\
\tab Suppose that $r(\theta)$ is an even function. Then $r$ has a cosine series expansion 
$$r(\theta)=\sum_{q=0}^{\infty}\wt{r}(q)\cos(q\theta)$$
where 
$$\wt{r}(q)=\frac{1}{2\pi}\int_{0}^{2\pi}r(\theta)\cos(q\theta)d\theta.$$
Furthermore, we have that the Fourier coefficients $\wh{r}(q)$ with respect to the $e^{iq\theta}$ basis are 
$$\wh{r}(q)=
\begin{cases}
\wt{r}(0) \ \ \ \ \  \text{if} \ q=0 \\
\wt{r}(|q|)/2 \ \text{if} \ q\neq 0 \\
\end{cases}
$$
in this situation, we have the following. 
\begin{corollary}
	If $r:\R/2\pi \to \R_{>0}$ is an even function satisfying the hypotheses of the previous proposition and cosine series coefficients defined as above, then 
	$$Z_{r}(s)=\wt{r^{2s}}(0)(E^{\sharp}.f_{0})(I,s)+\frac{1}{2}\sum_{q=1}^{\infty}i^{q}\wt{r^{2s}}(q)((E^{\sharp}.f_{q})(I,s)+(-1)^{q}(E^{\sharp}.f_{-q})(I,s)).$$
\end{corollary}

\section{Some Observations and First Calculations}

\tab This section is devoted to various points. First, we calculate the Hlawka Zeta Function of a circle. Second, we establish an action of $GL(2,\R)$ on the set of functions $C(\R/2\pi\Z,(0,\infty))$, which is key to some of our conjectures and calculation of the ellipse's Hlawka Zeta function later. Finally, we give the functional equation for the $(E^{\sharp}.f_{q})(I,s)$, to be used later as well. \\
\tab First, we have the following. 

\begin{lemma}
	Let $R(\theta)=cr(\theta)$, where $c$ is any positive real number. Then 
	$$Z_{R}(s)=c^{2s}Z_{r}(s).$$
\end{lemma}

\begin{proof}
	This is clear from the definition. 
\end{proof}

This proposition tells us that multiplying by a constant does not change $Z_{r}(s)$ in an analytically significant way. We will use this concept much of the time: it is often convenient to consider classes of functions up to a constant. This can be understood in the sense of point counting as well: from the formula, we see that if $A'_{r}(x)$ is the point count function as in the Point-counting formula, then 
$$A'_{r}(cx)=\frac{1}{2\pi i}\int_{2\sg+i\infty}^{2\sg-i\infty}Z_{r}(s)(cx)^{2s}\frac{ds}{s}=\frac{1}{2\pi i}\int_{2\sg+i\infty}^{2\sg-i\infty}(c^{2s}Z_{r}(s))x^{2s}\frac{ds}{s}$$
$$=\frac{1}{2\pi i}\int_{2\sg+i\infty}^{2\sg-i\infty}Z_{R}(s)x^{2s}\frac{ds}{s}=A'_{R}(x).$$
This is fairly intuitive from the geometry of the situation, but it is nice to see that intuition relayed in the zeta function. \\
\subsection*{The Case of the Circle}
\tab We will now consider the case of circles centered at the origin. The function $r$ associated to the circle of radius $c>0$ is the constant function $r(\theta)=c$. To establish notation, let $z=x+iy$, and define
$$E(z,s):=\frac{1}{2}\sum_{\gm \in \Gamma^{\infty}_{+}\backslash SL(2,\Z)} \Im(\gm z)^{s}$$
as the classical Eisenstein Series. Here we use the convention that an element of a matrix group multiplied by an element of $\C$ means the action by linear fractional transformations. With this notation, we have the following proposition. We also note that by a standard calculation,
$$E(z,s)=\frac{1}{2\zeta(2s)}\sum_{(m,n)\neq (0,0)}\frac{y^{s}}{|mz+n|^{2s}}.$$
We now are able to prove a functional equation for this simple case. 
\begin{proposition}
	For $r(\theta)=c$ we have that,
	$$Z_{r}(s)=2c^{2s}\zeta(2s)E(i,s)$$
	and furthermore 
	$$\frac{1}{c^{2s}}\Gamma(s)\pi^{-s}Z_{r}(s)=\frac{1}{c^{2(1-s)}}\Gamma(1-s)\pi^{-(1-s)}Z_{r}(1-s).$$
	Furthermore, $Z_{r}(s)$ has meromorphic continuation to all of $\C$ except at the points $s=1$, $s=0$. 
\end{proposition}

\begin{proof}
From the definitions and from the lemma, we have 
$$Z_{r}(s)=c^{2s}\sum_{(m,n)\neq (0,0)}\frac{1}{(m^{2}+n^{2})^{s}}=2c^{2s}\zeta(2s)E(i,s)$$
proving the first part of the proposition. We reference \cite[66-69]{Bump1} for the functional equation of the functional equation of the Eisenstein series here, which immediately gives the second part of the proposition. 
\end{proof}
 
This gives a functional equation $s\mapsto 1-s$ for this particular classical shape, and answering one of our guiding questions in this basic case. One can also understand this in terms of the Eisenstein series formula given in the previous section. One computes that for $r(\theta)=c$,

$$\wt{r^{2s}}(q)=\frac{1}{2\pi}\int_{0}^{2\pi} c^{2s}\cos(q\theta)d\theta=
\begin{cases}
0 \ \text{if} \ q\neq 0 \\
c^{2s} \ \text{if} \ q=0 \\
\end{cases}$$
and so the formula gives 
$$Z_{r}(\theta)=c^{2s}(E^{\sharp}.f_{0})(I,s)=c^{2s}\sum_{(m,n)\neq (0,0)}\frac{1}{(m^{2}+n^{2})^{s}}=c^{2s}\sum_{(m,n)\neq (0,0)}\frac{1}{(m^{2}+n^{2})^{s}}$$
which is a classical Eisenstein Series. 

\subsection*{Symmetries in the Eisenstein Series, Refining the Formula}

\tab Before attempting to calculate any other Hlawka Zeta functions, it is useful to begin discussing the third of our driving questions. We will begin in the analytic setting to discuss some of the symmetries in the Eisenstein series from before. 

\begin{lemma}
	Let 
	$\kappa_{\vph}=\begin{pmatrix}
	\cos(\vph) & \sin(\vph) \\
	-\sin(\vph) & \cos(\vph) \\
	\end{pmatrix}\in SO(2,\R)$ 
	as before. Then we have the following equality:
	$$(E^{\sharp}.f_{q})(\kappa_{\vph},s)=e^{iq\vph}(E^{\sharp}.f_{q})(I,s).$$ 
\end{lemma}

\begin{proof}
	We first would like to describe the effect of right multiplication on the right by an element of $SO(2,\R)$ on various pieces of the formula for the Eisenstein Series. Let $\gamma=
	\begin{pmatrix}
	a & b \\
	c & d \\
	\end{pmatrix}\in GL(2,\R)^{+}$.
	 In the Iwasawa decomposition of $\gamma\kappa_{\vph}=
	 \begin{pmatrix}
	 a' & b' \\ 
	 c' & d' \\
	 \end{pmatrix}
	 $, we have 
	$$\gamma\kappa_{\vph}=\begin{pmatrix}
	u & 0 \\
	0 & u \\
	\end{pmatrix}
	\begin{pmatrix}
	y^{1/2} &  xy^{-1/2} \\
	0 & y^{-1/2} \\
	\end{pmatrix} \kappa_{\theta}\kappa_{\vph}
	=\begin{pmatrix}
	u & 0 \\
	0 & u \\
	\end{pmatrix}
	\begin{pmatrix}
	y^{1/2} &  xy^{-1/2} \\
	0 & y^{-1/2} \\
	\end{pmatrix} \kappa_{\theta+\vph}$$
	and thus we have $\theta(c',d')=\theta(c,d)+\vph$. This gives that 
	$$(E^{\sharp}.f_{q})(\kappa_{\vph},s)=\sum_{(c,d)\neq (0,0)} \frac{e^{iq(\theta(c',d'))}}{(c^{2}+d^{2})^{s}}=\sum_{(c,d)\neq 0}\frac{e^{iq(\theta(c,d)+\vph)}}{(c^{2}+d^{2})^{s}}=e^{iq\vph}(E^{\sharp}.f_{q})(I,s)$$
\end{proof}

This allows us to pose a significant simplification to the $E^{\sharp}.f_{q}$'s in the formula. We list the simplifications here. 

\begin{proposition}
	We have the following.
	\begin{enumerate}
		\item $(E^{\sharp}.f_{q})(I,s)=0$ if $q$ is odd. 
		\item $(E^{\sharp}.f_{q})(I,s)=0$ if $q\equiv 2 \mod{4}$.  
		\item $(E^{\sharp}.f_{q})(I,s)=(E^{\sharp}.f_{q})(g,s)$ for $g\in SL(2,\Z)$. In particular, we have that $(E^{\sharp}.f_{q})(I,s)=(E^{\sharp}.f_{q})(\kappa_{k\pi/2},s)$ for $k\in \Z$. 
	\end{enumerate}
\end{proposition}

\begin{proof}
	For I, we see that for every $(c,d)\in \Z^{2}\ssm \{(0,0)\}$, there is a corresponding point $(-c,-d)$. These points are related by $\theta(c,d)=\theta(-c,-d)+\pi$ so the term associated with that point is
	$$\frac{e^{iq\theta(-c,-d)}}{((-c)^{2}+(-d)^{2})^{s}}=-\frac{e^{iq\theta(c,d)}}{(c^{2}+d^{2})^{s}}.$$
	The term corresponding to $(c,d)$ thus cancels with the one corresponding to $(-c,-d)$, so the sum telescopes to zero. \\
	\tab For II, we have that for every point $(c,d)$ there is a corresponding point $(-d,c)$, with $\theta(-d,c)=\theta(c,d)+\pi/2$. Thus the term corresponding to $(-d,c)$ is 
	$$\frac{e^{i(4q+2)\theta(-d,c)}}{(c^{2}+d^{2})^{s}}=\frac{e^{4qi\theta(c,d)}e^{i\pi}}{(c^{2}+d^{2})^{s}}=-\frac{e^{iq\theta(c,d)}}{(c^{2}+d^{2})^{s}}$$
	which cancels out the $(c,d)$ term in the sum and we have the result. \\
	\tab For III, this is immediate from the fact that translation by an element of a group is always a transitive action. Since $\kappa_{k\pi/2}\in SL(2,\Z)$, we have the other part of the proposition as well. 
\end{proof}

\tab We obtain immediately a refinement of the Eisenstein Series formula. 

\begin{corollary}[Eisenstein Series Formula, simplified version]
	Assume all of the hypotheses of the Eisenstein Series Formula. We have under the same conditions that 
	$$Z_{r}(s)=\sum_{q=-\infty}^{\infty}\wh{r^{2s}}(4q)(E^{\sharp}.f_{4q})(I,s)$$
	with this sum converging absolutely when $\Re(s)>1$. 
\end{corollary}

This corollary immediately implies the following as well.

\begin{corollary}[Simplified Eisenstein Series Formula for even functions]
	Let the definitions be as in Corollary 17. We have the following.
	\begin{equation}
		Z_{r}(s)=\wt{r^{2s}}(0)(E^{\sharp}.f_{0})(I,s)+\frac{1}{2}\sum_{q=0}^{\infty}\wt{r^{2s}}(4q)((E^{\sharp}.f_{4q})(I,s)+(E^{\sharp}.f_{-4q})(I,s))
	\end{equation}
\end{corollary} 

These corollaries, somewhat remarkably, tell us that only the Fourier coefficients of $r^{2s}$ which are indexed by multiples of 4 actually matter when computing the Hlawka Zeta function. \\
\tab This phenomenon is perhaps best understood not as evaluating the Eisenstein Series at a different matrix, but as rotation of the domain we are considering. To make this precise, we now would like to consider $R:\R \to \R_{>0}$ to be the unique $2\pi$ periodic extension of $r:\R/2\pi\Z \to \R_{>0}$ to the entire real line. If $f:\R\to \C$ is $2\pi$ periodic and $\vph\in \R/2\pi \Z$, let $f_{\vph}:\R\to \C$ be given by given by $f_{\vph}(\theta)=f(\theta+\vph)$. Recall the following lemma from harmonic analysis: 
\begin{lemma}
	Let $f:\R\to \C$ be a $2\pi$ periodic function and $\vph\in \R/2\pi \Z$. Then $e^{iq\vph}\wh{f}(q)$ is the $q$-th Fourier coefficient of the function $f_{\vph}$.
\end{lemma}

\begin{proof}
	We have 
	$$e^{iq\vph}\wh{f}(q)=\frac{1}{2\pi}\int_{0}^{2\pi}f(\theta)e^{-iq(\theta-\vph)}d\theta=\frac{1}{2\pi}\int_{0}^{2\pi}f(\theta+\vph)e^{-iq\theta}d\theta,$$
	which is the desired equality. 
\end{proof}

This lemma and the prior discussion have the following consequence. 

\begin{proposition}
	Let $r$ and $R$ be as above where $r$ satisfies the hypothesis for the Eisenstein Series Formula. Let $r_{\vph}$ be the function defined on $\R/2\pi\Z$ corresponding to $R_{\vph}$. We have that 
	$$Z_{r_{\vph}}=\sum_{q=-\infty}^{\infty}\wh{r^{2s}}(4q)(E^{\sharp}.f_{4q})(\kappa_{\vph},s).$$
	Moreover, if $\vph=\frac{k\pi}{2}$ for $k\in \Z$, then $Z_{r}=Z_{r_{\vph}}$. 
\end{proposition}

\begin{proof}
	From Lemma 20, Proposition 21, and Lemma 24, the result immediately follows.  
\end{proof}

\tab This gives a first guess at a possible fibers for the Hlawka Zeta function as described in the first section: the subgroup of rotations in $D_{8}$, the dihedral group of order 8. In this light, we have a subgroup isomorphic to this group embedded in $SL(2,\Z)$: the group $\{I,\kappa_{\pi/2},\kappa_{\pi},\kappa_{3\pi/2}\}$, which is embedded in the maximal compact group $K=SO(2,\R)$ of rotations. One important note is that of course $O(2,\Z)$ is isomorphic to $D_{8}$, but the reflections are generated by $R=\begin{pmatrix}
-1 & 0 \\
0 & 1
\end{pmatrix}$, which can be seen to send $r\mapsto -r$ in the parametrization: of course then $-r<0$, which we try to avoid. We note that there is a natural action of $K$ on the set of $r$'s given by $(\kappa_{\vph}.r)(\theta)=r_{\vph}(\theta)$. What this proposition tells us is that this action corresponds with putting a $\kappa_{\vph}$ in the input of the Eisenstein Series, or vise versa. \\
\subsection*{Action of $GL(2,\R)$ on Functions}
\tab One may expand this view quite a bit, as there is a nice geometric interpretation of this as well. There is an action of $GL(2,\R)$ on $r$ given by the action of $GL(2,\R)$ on the graph of $r$. We will define this action precisely now. This will give an action of $GL(2,\R)$ on the Hlawka Zeta function. Throughout, we use a fact from geometry: if $D\subset \R^{2}$ is a star-shaped region with respect to the origin, then applying a non-singular linear transformation $g\in GL(2,\R)$ to this domain yields another star-shaped region with respect to the origin. \\
\tab First, let us establish some notation. Take $C(\R/2\pi\Z,(0,\infty))$ to be the set of continuous functions from $\R/2\pi\Z$ to $(0,\infty)$. Define the map $X:(0,\infty)\times \R/2\pi\Z \to \R^{2}\ssm\{(0,0)\}$ to be 
$$X(c,\vph)=\begin{pmatrix}
c\cos(\vph)\\
c\sin(\vph)
\end{pmatrix}$$
where $c\in (0,\infty)$. Also, define the map $\theta:\R^{2}\ssm \{(0,0)\}\to \R/2\pi\Z$ such that $\theta(v)$ is the angle between the positive $x$ axis and the ray from the origin to $v\in \R^{2}\ssm\{(0,0)\}$. Some initial observations about these maps are 
\begin{enumerate}
	\item First, $\theta(X(c,\vph))=\vph$, so $\theta$ undoes $X$ on the second component, and so $\theta(X(c,\vph))$ is independent of $c$. \\
	\item The map $X$ is a smooth diffeomorphism, and $\theta$ is surjective. In fact the restriction $\theta|_{\gm}$ is a smooth diffeomorphism as long as $\gm$ is a continuous closed curve which is the boundary of a star-shaped region with respect to the origin. Moreover, we have that if $\gm$ is given in polar coordinates as $\gm:\R/2\pi\Z \to (0,\infty)$, then $X(\gm(\vph),\theta|_{\gm}(\gm(\vph)\cos(\vph),\gm(\vph)\sin(\vph)))$ traces the graph of the curve $\gamma$. \\
	\item For all $g\in GL(2,\R)$ and $u,v\in \R^{2}\ssm\{(0,0)\}$, $\theta(v)=\theta(u)$ if and only if $\theta(gv)=\theta(gu)$. This is simply due to the fact that $g$ is a linear map. 
\end{enumerate}

\tab Let $\theta_{g,\gm}:\R/2\pi\Z\to \R/2\pi\Z$ be the map given by 
$$\theta_{g,\gm}(\vph)=\theta(gX(\gm(\vph),\vph))$$
where $\gm\in C(\R/2\pi\Z,(0,\infty))$. In fact, $\theta_{g,\gm}=\theta_{g,1}$ by observations above, but it will be crucial later that we have the auxiliary $\gm$.  

\begin{lemma}
	The map $\theta_{g,\gm}$ is a diffeomorphism for all $g\in GL(2,\R)$, $\gm\in  C(\R/2\pi\Z,(0,\infty))$. 
\end{lemma}
\begin{proof}
It suffices to consider $\gm=1$. We see that $\theta_{g}$ the composition $\theta\circ g\circ  X|_{\{1\}\times \R/2\pi\Z}$, where $g$ is simply considered as a linear map. It is readily seen that $X|_{\{1\}\times \R/2\pi\Z}$ is a diffeomorphism of $\{1\}\times \R/2\pi\Z$ onto $S^{1}$. Now, $g$ maps $S^{1}$ diffeomorphically onto the boundary of some ellipse $E_{g}$, and since this is star-shaped with respect to the origin, $\theta|_{\partial E_{g}}$ is a diffeomorphism.   
\end{proof}

Moreover, we also have:

\begin{lemma}
	If $g,h\in GL(2,\R)$, then $\theta_{g,1}\circ \theta_{h,1}=\theta_{gh,1}$
\end{lemma}
\begin{proof}
	We will change $\theta_{g,1}$ to $\theta_{g,\epsilon_{h}}$, where $\epsilon_{h}$ is the (polar) parameterization of the ellipse determined by $h\cdot S^{1}$ in $\R^{2}$ which begins at the point on the positive $x$ axis. We use the same idea as before: this composition can be broken down into 
	$$\theta \circ g \circ X|_{\epsilon_{h}(\R/2\pi/Z)\times \R/2\pi\Z}\circ \theta \circ h\circ  X|_{\{1\}\times \R/2\pi\Z}.$$
	It suffices to show that the maps  $X|_{\epsilon_{h}(\R/2\pi/Z)\times \R/2\pi\Z}$ and $\theta$ compose to the identity. The input here is the graph of $\epsilon_{h}$. Now,
	$$X|_{\epsilon_{h}(\R/2\pi/Z)\times \R/2\pi\Z}\circ \theta(\epsilon_{h}(\vph)\cos(\vph),\epsilon(\vph)\sin(\vph))= X|_{\epsilon_{h}(\R/2\pi/Z)\times \R/2\pi\Z}(\theta(\epsilon_{h}(\vph)\cos(\vph),\epsilon(\vph)\sin(\vph)))$$
	$$=X(\epsilon_{h}(\vph),\theta(\epsilon_{h}(\vph)\cos(\vph),\epsilon(\vph)\sin(\vph))),$$ which we have already noted is the graph of $\epsilon_{h}$.
	Now since $g\circ h=gh$ by definition, we have proven the lemma.
\end{proof}
From here, we will drop the $\gm$ from the $\theta_{g,\gm}$ notation without loss of generality. Now, we may state the following proposition. 

\begin{proposition}
	For all $r\in C(\R/2\pi\Z,(0,\infty))$, and $g\in GL(2,\R)$, there is a unique map $g\cdot r\in C(\R/2\pi\Z,(0,\infty))$ with the property that 
	$$gX(r(\vph),\vph)=X((g\cdot r)(\theta_{g}(\vph)),\theta_{g}(\vph))$$
	for all $\vph\in \R/2\pi\Z$. Moreover, the assignment $(g,r)\mapsto g\cdot r$ is a group action of $GL(2,\R)$ on $C(\R/2\pi\Z,(0,\infty))$. 
\end{proposition}
\begin{proof}
	It is clear that $\theta(gX(r(\vph),\vph))=\theta_{g}(\vph)$ by definition. Now, $|gX(r(\vph),\vph)|$ is a function of $\vph$, and is continuous since $r$ is. Now, since $\theta_{g}$ is a diffeomorphism, we may write $|gX(r(\vph),\vph)|$ in $\theta_{g}(\vph)$ coordinates by substituting $\theta_{g^{-1}}(\theta_{g}(\vph))$ for $\vph$. This gives a continuous function $$|gX(r(\theta_{g^{-1}}(\theta_{g}(\vph))), \theta_{g^{-1}}(\theta_{g}(\vph)))|=(g\cdot r)(\theta_{g}(\vph))$$ in terms of $\theta_{g}(\vph)$.   \\
	\tab To show that this is a group action, note first that for the identity $I\in GL(2,\R)$ we have $\theta_{I}(\vph)=\vph$, so that $IX(r(\vph),\vph)=X(r(\vph),\vph)$ so $r=I\cdot r$. Let $g,h\in GL(2,\R)$. On one hand,
	$$ghX(r(\vph),\vph)=X((gh\cdot r)(\theta_{gh}(\vph)),\theta_{gh}(\vph))$$
	and on the other 
	$$g(hX(r(\vph),\vph))=gX((h\cdot r)(\theta_{h}(\vph)),\theta_{h}(\vph))=X((g\cdot (h\cdot r))(\theta_{g}\theta_{h}(\vph)), \theta_{g}\theta_{h}(\vph))$$
	$$=X((g\cdot (h\cdot r))(\theta_{gh}(\vph)),\theta_{gh}(\vph) )$$
	so that $gh\cdot r=g\cdot(h\cdot r)$. 
\end{proof}

\tab In general, computing $\theta_{g}$ for a random element of $GL(2,\R)$, and consequently $g\cdot r$, can be rather difficult. However, we have already seen that for rotation matrices $\kappa_{\vph}$, this re-parameterization is just $\theta_{\kappa_{\vph}}(\theta)=\theta+\vph$ and $(\kappa_{\vph}\cdot r)=r$. Furthermore, this can be done in general using the following decomposition.

\begin{proposition}[Cartan Decomposition]
	Let $G=GL(2,\R)^{+}$ (or $GL(2,\R)$) and $K=SO(2,\R)$ (or $K=O(2,\R)$). Every double coset in $K\backslash G / K$ has a unique representative 
	$$\begin{pmatrix}
	d_{1} & 0 \\
	0 & d_{2} 
	\end{pmatrix}$$
	with $d_{i}\in R$ and $d_{1}\geq d_{2} > 0$. 	
\end{proposition}

\begin{proof}
	We refer to \cite{Bump1} for the proof.   
\end{proof}

This allows us to write every element of $GL(2,\R)^{+}$ as 
$$g=\kappa_{\vph_{1}}
\begin{pmatrix}
d_{1} & \\
0 & d_{2} 
\end{pmatrix}
\kappa_{\vph_{2}}$$
where $\kappa_{\vph_{i}}\in SO(2,\R)$. A further simplification yields 
$$g=d_{2}\kappa_{\vph_{1}}
\begin{pmatrix}
	a & 0 \\
	0 & 1 
\end{pmatrix}
\kappa_{\vph_{2}}$$
where $d_{2}$ is a scalar matrix with $d_{2}$ as the scalar and $a=d_{1}/d_{2}>1$. This means that computing the action of $g$ on $r$ is tantamount to finding the action of $\bar{a}=\begin{pmatrix}
a & 0 \\
0 & 1 \\
\end{pmatrix}$ on $r$, since the other actions are already well understood. This can be done with some trigonometry We first compute $$\theta_{\bar{a}}(\vph)=\begin{cases}
\arctan(\sin(\vph)/a\cos(\vph)) & \vph\in [0,\pi/2]\cup [3\pi/2,2\pi] \\
\pi+\arctan(\sin(\vph)/a\cos(\vph)) & \vph\in [\pi/2,3\pi/2]
\end{cases}$$
Now the length of a vector in $\bar{a}X(r(\vph),\vph)$ is $$r(\vph)(a^{2}\cos^{2}(\vph)+\sin^{2}(\vph))^{1/2}$$
and writing this in $\theta_{\bar{a}}(\vph)=\vph'$ coordinates, 
$$(\bar{a}\cdot r)(\vph')=
\begin{cases}
r\left(\arctan\left(\frac{a\sin(\vph')}{\cos(\vph')}\right)\right)\left(\frac{a^{2}}{a^{2}+(1-a^{2})\cos^{2}(\vph')}\right)^{1/2} & \vph'\in [0,\pi/2]\cup [3\pi/2,2\pi]  \\
r\left(\arctan\left(\frac{a\sin(\vph')}{\cos(\vph')}\right)+\pi\right)\left(\frac{a^{2}}{a^{2}+(1-a^{2})\cos^{2}(\vph')}\right)^{1/2} & \vph'\in [\pi/2,3\pi/2].
\end{cases}$$
Notably, if $r$ is the constant function, these two pieces are equal. This will be used later.\\
\tab Of course, in our equations for the Hlawka Zeta function, one needs to compute the Fourier Coefficients for $r^{2s}$, and so to compute $Z_{(\bar{a}\cdot r)}(s)$ becomes a computationally very difficult, depending on the $r$ in question. We essentially do this for when $r$ is a constant function later in the paper. Currently, we only have the following conjecture.

\begin{conjecture}
	Let $r\in C(\R/2\pi\Z,(0,\infty))$ be a non constant function where infinitely many of the Fourier coefficients $\hat{r}(n)$ are non-zero. Then $Z_{r}(s)=Z_{g\cdot r}(s)$ if and only if $g\in O(2,\Z)$, i.e. $g$ is a isometry of the lattice $\Z^{2}$. 
\end{conjecture}

Note that if the condition on Fourier coefficients is dropped, then clearly $r$ could have a period of less than $2\pi$, so that $\kappa_{\vph}$ would stabilize $r$ for values other that $\frac{k\pi}{2}$.

\subsection*{Functional Equation for the Eisenstein Series in Question}

\tab In searching for the functional equation for the Hlawka Zeta Function, we would like to somehow leverage the fact that each of the Eisenstein series that are in the infinite sum have a functional equation. We look at these Eisenstein series separately for a moment to do this. Most of this is pulled from \cite{Bump1}, Section 3.7, and will need to use adelic methods. For this reason, some of the finer points are left out. \\
\tab Let $\A$ be the adele ring of $\Q$, and let $\chi_{i}$ for $i\in\{1,2\}$ be two quasichararcters of $\A^{\times}/\Q^{\times}$. Then the space $V_{\chi_{1},\chi_{2}}$ is defined to be all smooth and $K$-finite (where $K$ is the maximal compact subgroup of $GL(2,\A)$ defined in 3.3 of \cite{Bump1}) functions on $GL(2,\A)$ which satisfy 
$$f\left( \begin{pmatrix}
y_{1} & x \\
0 & y_{2}
\end{pmatrix}g\right)=|y_{1}|^{1/2}|y_{2}|^{-1/2}\chi_{1}(y_{1})\chi_{2}(y_{2})f(g).$$
This is similar to our definition of $f_{q}(g,s_{1},s_{2})$. In fact, we see that if $\chi_{i}(y)=|y|^{s_{i}}$, then for $g\in GL(2,\R)^{+}$, $f_{q}$ satisfies these conditions if $q$ is even. This assumption is reasonable since the Eisenstein series which we will eventually define are zero on odd $q$ anyways, and we will henceforth assume this. We may construct a function $F_{q}:GL(2,\A)\to \C$ which agrees with $f_{q}$ on $GL(2,\R)^{+}$ in the following way. Since $K=O(2,\R)\times \prod_{p} K_{p}$ where $K_{p}=GL(2,\Z_{p})$, by the Iwasawa decomposition it suffices to give $F_{q}(k)$ for $k\in K$, but this is reduced to giving
$$F_{q}(k)=F_{q,\infty}(R^{i}\kappa_{\vph})\cdot \prod_{p} F_{q,p}(k_{p})$$
for $R=\begin{pmatrix}
-1 & 0 \\
0 & 1 
\end{pmatrix}$, $\kappa_{\vph}\in O(2,\R)$, and $k_{p}\in K_{p}$. Here we may set $F_{q,p}$ to be the spherical vector normalized so that $F_{q,p}(k_{p})=1$ for all $k_{p}\in K_{p}$ for all $p$ prime places, and $F_{q,\infty}(R^{i}\kappa_{\vph})=(-1)^{i}f_{q}(\kappa_{\vph})$. Thus 
$$F_{q}(k)=(-1)^{i}f_{q}(\kappa_{\vph})$$
for all $k\in K$, so for $g\in GL(2,\R)^{+}$, $F_{q}(g)$ agrees with $f_{q}(g)$. We note that the $K$-finiteness and smoothness then follow from that of $f_{q}$. \\
\tab We now define the Eisenstein Series for $f$ in this more general context. 
\begin{definition}
	For $f\in V_{\chi_{1},\chi_{2}}$, define the \underline{Eisenstein Series acting on $f$} to be 
	$$(E.f)(g)=2\sum_{\gamma\in B(Q)\backslash GL(2,\Q)}f(\gamma g)$$
	for $g\in GL(2,\A)$.  
\end{definition} 
We note that for any element $\gamma\in GL(2,\R)$, there exists $\beta\in B(\Q)$ such that $\beta \gamma \in SL(2,\Z)$. Thus we may reindex this sum as 
$$(E.f)(g)=2\sum_{\gamma\in \Gamma^{\infty}\backslash SL(2,\Z)}f(\gamma g).$$
where $\Gamma^{\infty}=SL(2,\Z)\cap B(\Q)$. This matches our definition from before of $(E.f_{q})(g,s_{1},s_{2})$, i.e. $(E.f_{q})(g,s_{1},s_{2})=(E.F_{q})(g)$. \\
\tab There a couple of things to note in this case about the $L$-functions and Gamma functions: namely, that in the notation of \cite{Bump1}, our $\xi_{i}$ are trivial so that the complete $L$-function $L(2s,\xi_{1}\xi_{2}^{-1})$ has $L_{\infty}(2s,1)=\pi^{-s}\Gamma(s)$ at the real place and
$$L(2s,1)=\pi^{-s}\Gamma(s)\prod_{p\text{ prime}} (1-p^{-2s})^{-1}=\pi^{-s}\Gamma(s)\zeta(2s).$$
Furthermore, the Gamma factor $\gamma_{\infty}(s,1,\psi_{\infty})=\pi^{s-1/2}\frac{\Gamma((1-s)/2)}{\Gamma(s/2)}$. \\
\tab It is now possible to state the following precisely: 
\begin{theorem}
	The Eisenstein series $(E^{\sharp}.f_{q})(g,s)$ satisfies the following functional equation
	\begin{equation}
		\pi^{-s}\Gamma(s)(E^{\sharp}.f_{q})(g,s)=\pi^{-(1-s)}\Gamma(1-s)\frac{\Gamma(s)^{2}(-i)^{q}}{\Gamma(s+q/2)\Gamma(s-q/2)}(E^{\sharp}.f_{q})(g,1-s)
	\end{equation}
\end{theorem}

\begin{proof}
	We cite \cite{Bump1} for most of the heavy lifting in this proof. There a couple of things to note in this case about the $L$-functions and Gamma functions: namely, that in the notation of \cite{Bump1}, our $\xi_{i}$ are trivial so that the complete $L$-function $L(2s,\xi_{1}\xi_{2}^{-1})$ has $L_{\infty}(2s,1)=\pi^{-s}\Gamma(s)$ at the real place and
	$$L(2s,1)=\pi^{-s}\Gamma(s)\prod_{p\text{ prime}} (1-p^{-2s})^{-1}=\pi^{-s}\Gamma(s)\zeta(2s).$$
	Furthermore, the Gamma factor $\gamma_{\infty}(s,1,\psi_{\infty})=\pi^{s-1/2}\frac{\Gamma((1-s)/2)}{\Gamma(s/2)}$.
	Setting 
	$$\widetilde{F_{q}}_{\nu}=
	\begin{cases}
	\widetilde{F}^{\circ}_{p_{\nu}} & \nu \text{ a prime place} \\
	\frac{\pi^{-s}\Gamma(s)}{\pi^{(2-2s)-1/2}\Gamma(s-\frac{1}{2})\Gamma(1-s)^{-1}\pi^{-(1-s)}\Gamma(1-s)}M(s)f_{q} & \nu \text{ not a prime place,}
	\end{cases}$$
	where $$M(s)f_{q,s}=(-i)^{q}\pi^{1/2}\frac{\Gamma(s)\Gamma(s-\frac{1}{2})}{\Gamma(s+q/2)\Gamma(s-q/2)}f_{q,1-s}$$
	coming from Proposition 2.6.3 of \cite{Bump1}. Thus the real place of $\widetilde{F}_{q}$ is 
	$$\frac{\Gamma(s)^{2}(-i)^{q}}{\Gamma(s+q/2)\Gamma(s-q/2)}f_{q,1-s}$$ 
	Furthermore, setting
	 $$E^{\sharp}(g,F_{q})=L(2s,1)E(g,F_{q})=\pi^{-s}\Gamma(s)\zeta(2s)E(g,f),$$
	 Theorem 3.7.2 of \cite{Bump1} states 
	 $E^{\sharp}(g,F_{q})=E^{\sharp}(g,\widetilde{F_{q}})$, where $E^{\sharp}(g,\widetilde{f})=L(2-2s,1)E(g,\widetilde{f})$ so that 
	 $$\pi^{-s}\Gamma(s)(E^{\sharp}.f_{q})(g,s)=\pi^{-(1-s)}\Gamma(1-s)\frac{\Gamma(s)^{2}(-i)^{q}}{\Gamma(s+q/2)\Gamma(s-q/2)}(E^{\sharp}.f_{q})(g,1-s)$$
	 as desired.
\end{proof}
It should be noted that if $q=0$, this gives the functional equation for the classical Eisenstein Series. 

\section{The Case of Ellipses}
 
\tab We would now like to study the Hlawka Zeta Function of ellipses in the plane. The main idea is that Hlawka Zeta function of ellipses are understood to be either Epstein Zeta functions, $(E.f_{0})(g,s_{1},s_{2})$ for a particular $g$, or $E(z,s)$ for a certain $z$. This gives a functional equation $s\mapsto 1-s$, which can be interpreted in terms of the $\wh{r^{2s}}(4q)$ as combinatorial identities of a sort. 
\subsection*{Finding a Functional Equation}
\tab Consider the following situation: we have an ellipse $e$ with center at the origin, major axis length $a$, minor axis length $b$, and is rotated $\vph$ degrees counterclockwise about the origin. This is the image of the unit circle under the shear transformation $g_{e}=\kappa_{-\vph}\begin{pmatrix}
a & 0\\
0 & b
\end{pmatrix}$. This means the equation of an ellipse is given by $|g_{e}^{-1}\boldsymbol{x}|^{2}=1$. We would like to calculate the Hlawka Zeta function for such an object. Let us first consider 
\begin{definition}
	Let $u \in GL(n,\R)$ be a positive definite symmetric matrix. Then the \underline{Epstein Zeta Function} is 
	$$\mathcal{E}(u,s)=\sum_{ x\in \Z^{n}\ssm\{0\}}(x^{T}ux)^{-s}.$$
\end{definition}
It is a theorem of linear algebra that every positive definite symmetric form $u$ can be written as $g^{T}g$ for some $g\in GL(2,\R)$. Thus the expression in the sum is 
$$(x^{T}g^{T}gx)^{-s}=((gx)^{T}(gx))^{-s}=|gx|^{-2s}$$
for some $g\in GL(2,\R)$. In the notation of Section 1, if $D$ is the ellipse $e$, then the $t(m,n)$'s for the ellipse are $t(m,n)^{2}=|g_{e}^{-1}(m,n)|^{2}$. Summing over all $(m,n)$, we get
\begin{proposition}
	Set $g_{e}$ as before. The Hlawka Zeta function for an ellipse determined by $g_{e}$ is 
	$Z(r,s)=\mathcal{E}(^{T}g_{e}^{-1}g_{e}^{-1},s)$.
\end{proposition}

\begin{proof}
	Clear.
\end{proof}

Now, from \cite{Taylor1}, the Epstein Zeta function has a functional equation 
$$\pi^{-s}\Gamma(s)\mathcal{E}(u,s)=\pi^{-(1-s)}\Gamma(1-s)\det(u)^{-1/2}\mathcal{E}(u^{-1},n/2-s)$$
for $u\in GL(n,\R)$. Then the Hlawka Zeta function also has a functional equation 
\begin{equation}
\pi^{-s}\Gamma(s)Z(r,s)=\pi^{-(1-s)}\Gamma(1-s)(ab)^{-1/2}Z(r^{\ast},1-s)
\end{equation}
Where $r^{\ast}$ is the ellipse determined by $^{T}g_{e}^{-1}$ in the same way as before. What we would now like to do is interpret this functional equation in terms of the formula in terms of Eisenstein Series. 
\subsection*{Interpreting the Functional Equation}

From the previous section, we know that the action of $GL(2,\R)$ on our polar functions $r\in C(\R/2\pi\Z,(0,\infty))$ is essentially determined by the Cartan decomposition of $g$. For $g_{e}$, this decomposition is clearly $g_{e}=b\kappa_{-\vph}\begin{pmatrix}
a/b & 0 \\
0 & 1 \\
\end{pmatrix}$. Thus the action on the constant function, i.e. the unit circle, $r=1$ of $g_{e}$ is 
$$(g_{e}.r)(\theta)=ab\left(\left(\frac{a}{b}\right)^{2}+\left(1-\left(\frac{a}{b}\right)^{2}\right)\cos^{2}(\theta-\vph)\right)^{-1/2}.$$
We would like to find the Fourier coefficients of $(g_{e}.r)(\theta)^{2s}$ for any $s$. Since we know the constant $a$ will simply become $a^{s}$ and the action by $\kappa_{-\vph}$ is known, it suffices to find the Fourier coefficients for 
$$\left(\frac{a}{b}\right)^{2}+\left(1-\left(\frac{a}{b}\right)^{2}\right)\cos^{2}(\theta).$$
First, set $\frac{a^{2}}{b^{2}}=c>0$, and $d=1-c$. Thus we have reduced the problem to finding the Fourier Coefficient of 
$$E(\theta)=(c+d\cos^{2}(\theta))^{s}.$$ 
Notably, this is even, so we only need to know $\int_{0}^{2\pi}E(\theta)\cos(q\theta)d\theta$. We begin by recalling the extended binomial theorem. For $|x|>|y|>0$ and $r\in \C$, we have 
$$(x+y)^{r}=\sum_{k=0}^{\infty}{r\choose k}x^{r-k}y^{k}$$
where ${r\choose k}=\frac{r(r-1)\cdots (r-k+1)}{k!}=\frac{\Gamma(r+k)}{k!\Gamma(r)}$.
In this case, $$\left|\frac{a^{2}}{b^{2}}\right|>\left|\frac{a^{2}}{b^{2}}-1\right|\geq \left|\left(1-\frac{a^{2}}{b^{2}}\right)\cos^{2}(\theta)\right|,$$
since $a>b$. Now, 
$$(c+d\cos^{2}(\theta))^{s}=\sum_{k=0}^{\infty} {s\choose k}c^{s-k}d^{k}\cos^{2k}(\theta).$$
We evaluate 
$$\cos^{2k}(\theta)=(\Re(e^{i\theta}))^{2k}=\frac{1}{2^{2k}}(e^{i\theta}+e^{-i\theta})^{k}=\frac{1}{2^{2k-1}}\sum_{i=0}^{k}{2k\choose i}\cos(2(k-i)\theta).$$
Here we would like to isolate a cosine term with $4q\theta$ for some $q\in \Z$, since when we integrate the rest of the terms vanish. This means $2q=k-i$ so $i=k-2q>0$. Thus the term with $\cos(4q\theta)$ is 
$$\sum_{k>2q} {s \choose k} {2k \choose k-2q} \frac{c^{s-k}d^{k}}{2^{2k-1}}\cos(4q\theta).$$ 
Thus, upon integrating, we obtain that the $4q$-th Fourier coefficient is 
	$$c^{s}\sum_{k>2q} {s \choose k} {2k \choose k-2q} \frac{c^{-k}d^{k}}{2^{2k-1}}.$$
Reverting to $s\mapsto -s$, we obtain the necessary coefficients
\begin{equation}
	\widetilde{(g_{e}.1)^{2s}}(4q)=(ab)^{2s}c^{-s}\sum_{k>2q} {-s \choose k} {2k \choose k-2q} \frac{c^{-k}d^{k}}{2^{2k-1}}=c^{-s}\sum_{j=0}^{\infty}{2j+4q \choose j}\frac{d^{j+2q}\Gamma(j+2q-s)}{2^{2q-1}c^{j+2q}\Gamma(-s)}(1/2)^{j}.
\end{equation}
in the expression
$$	Z((g_{e}.1),s)=\wt{(g_{e}.1)^{2s}}(0)(E^{\sharp}.f_{0})(\kappa_{-\vph},s)+\frac{1}{2}\sum_{q=0}^{\infty}\wt{(g_{e}.1)^{2s}}(4q)((E^{\sharp}.f_{4q})(\kappa_{-\vph},s)+(E^{\sharp}.f_{-4q})(\kappa_{-\vph},s))$$
Now, we want to relate coefficients in the Eisenstein Series, so that by multiplying $Z(r,s)$ by $\pi^{-s}\Gamma(s)$, we evoke the functional equation just derived for the $E^{\sharp}.f_{q}$'s in each of the coefficients. We have that that $^{T}g_{e}^{-1}=\kappa_{-\vph}\begin{pmatrix}
1/a & 0 \\
0 & 1/b
\end{pmatrix}$ will yield the same ellipse as $\kappa_{-\vph+\pi/2}\begin{pmatrix}
1/b & 0 \\
0 & 1/a \\
\end{pmatrix}$. This means that $c=a/b=(1/b)/(1/a)$ stays the same along with $d$. Now, the only other difference is in the $\kappa_{-\vph+\pi/2}$, but by our lemma from the previous section, $(E^{\sharp}.f_{4q})(\kappa_{-\vph+\pi/2},s)=(E^{\sharp}.f_{4q})(\kappa_{-\vph},s)$, so this works out. Thus, the only change in variables is $s\mapsto 1-s$ when we multiply the above formula for the Hlawka Zeta Function by $\pi^{-s}\Gamma(s)$. Thus we arrive at the following theorem. 
\begin{theorem}
	For all $q>0$, we have that 
	$$
\pi^{-s}\Gamma(s)(ab)^{2s}c^{-s}\sum_{j=0}^{\infty}{2j+4q \choose j}\frac{d^{j+2q}\Gamma(j+2q-s)}{2^{2q-1}c^{j+2q}\Gamma(-s)}(1/2)^{j}$$
\begin{equation}
=\pi^{-(1-s)}\Gamma(1-s)\frac{\Gamma(s)^{2}}{\Gamma(s+2q)\Gamma(s-2q)}(ab)^{2s-3/2}c^{-(1-s)}\sum_{j=0}^{\infty}{2j+4q \choose j}\frac{d^{j+2q}\Gamma(j+2q-1+s)}{2^{2q-1}c^{j+2q}\Gamma(-1+s)}(1/2)^{j}.
	\end{equation}
\end{theorem}

This is a combinatorial identity in terms of two seemingly unrelated power series. In fact, both of the sums are related in that they may both be interpreted as hypergeometric functions, i.e. functions of the form
$$_{2}F_{1}(a,b;c;z)=\sum_{n=0} \frac{(a)_{n}(b)_{n}}{(c)_{n}}\frac{z^{n}}{n!}$$
defined for $z\in \C$ with $|z|<1$, and the notation $(a)_{n}$ means $a(a+1)\cdots (a+n-1)$ if $n>0$ and $1$ if $n=0$ (This is the so called rising Pochhammer Symbol).  
 
\section{The Case of the Square}

It is might be reasonable to ask where the Hlawka Zeta function has a functional equation for general $r\in C(\R/2\pi\Z,(0,\infty))$. For example, let 
\begin{equation}
Z^{\ast}_{r}(s)=\pi^{-s}\Gamma(s)Z_{r}(s).\end{equation}
Then, if we follow the case of ellipses, $Z_{r}^{\ast}(s)$ may meromorphic continuation to all of $\C$ and satisfies the functional equation. 
$$Z_{r}^{\ast}(s)=Z_{r}^{\ast}(1-s).$$
This informal conjecture, motivated by the example of the square and other Modular forms, breaks down in the case of the square. One can fairly simply compute the Hlawka Zeta function for the square centered at the origin with side length 2. Here, one may see that $t_{k}=k$, and $a_{k}=8k$. Thus the Hlawka Zeta function is 
$$Z_{r}(s)=\sum_{k=1}^{\infty}\frac{8k}{k^{2s}}=8\zeta(2s-1).$$
Now, from the classical functional equation of the zeta function, $Z_{r}(s)$ has a functional equation $s\mapsto 3/2-s$. A further question might ask whether in some sense $Z_{r}(s)$ has a functional equation $s\mapsto 1-s$ in a more general setting.  \\
\tab It is natural to consider a larger classes of functional equation. We will say that a complex function $f$ defined in a domain $D=\{s|\Re(s)>c,c\in \R\}$ has a \underline{regular functional equation} for $s\mapsto 1-s$ if $f$ has meromorphic continuation to the entire plane which satisfies 
$$f(s)=\frac{A^{1-s}\prod_{i=1}^{k} \Gamma(\alpha_{i}(1-s)+\mu_{i})}{B^{s}\prod_{j=1}^{K} \Gamma(\beta_{j}s+\om_{j})}f(1-s)$$
with $\alpha_{i},\beta_{j}$ real and positive, and $A,B,\mu_{i},\om_{j}$ complex. This is a bit of a generalization of the functional equation condition of the "Selberg Class" discussed in the original paper by Selberg in \cite{Selberg1}. \\
\tab In our case, as for $L$ functions and other modular forms, it seems reasonable to assume that $Z_{r}(1-s)$ should be replaced by $Z_{r^{\ast}}(1-s)$, where $r^{\ast}$ is another continuous function. Now, we have a well posed question: If $r:\R/2\pi\Z\to \R_{>0}$ traces the square, is there a continuous function $r^{\ast}:\R/2\pi\Z\to \R_{>0}$ such that
\begin{equation}
	Z_{r}(s)=\frac{A^{1-s}\prod_{i=1}^{k} \Gamma(\alpha_{i}(1-s)+\mu_{i})}{B^{s}\prod_{j=1}^{K} \Gamma(\beta_{j}s+\om_{j})}Z_{r^{\ast}}(1-s)
\end{equation}
for some $\alpha_{i},\beta_{j}\in \R_{>0}$, $A,B,\mu_{i},\om_{j}\in \C$, and for all $s\in \C$? We will give a negative answer to this question. \\
\tab In our other question about the fibers of the mapping $r\mapsto Z_{r}(s)$, we have at least one conjecture, mentioned in a previous section. However, this conjecture only applies to the $GL(2,\R)$-orbits of an element $r\in C(\R/2\pi\Z,(0,\infty))$. There is a broader conjecture looming here, namely whether $Z_{r}(s)=Z_{r'}(s)$ implies that $r'$ is in the $GL(2,\R)$-orbit of $r$. The square provides a counterexample to this conjecture. 

\subsection*{Functional Equation}
\tab The goal of this section is to prove the following.
\begin{theorem}
If $r$ traces out the unit square, there is no continuous $r^{\ast}$ such that (12) holds for some $\alpha_{i},\beta_{j}\in \R_{>0}$, $A,B,\mu_{i},\om_{j}\in \C$, and for all $s\in \C$. 
\end{theorem}

We will first need a couple of lemmas here, which was only proven partially in the first section. However, there is a difference in notation between ours and the more standard notation. Let $\{a_{n}\}$, $\{\lambda_{n}\}$ be two sequences of real numbers, with $\lambda_{n}$ increasing. Then a \underline{Generalized Dirichlet Series} is of the form 
$$f(s)=\sum_{n=1}^{\infty}a_{n}e^{-\lambda_{n}s}.$$
This is an equivalent definition to our $t$-Dirichlet Series given by 
$$(D_{t}.a)(s)=\sum_{n=1}^{\infty}a_{n}t_{n}^{-s}=\sum_{n=1}^{\infty}a_{n}e^{-\log(t_{n})s}$$
and there is a change of variables $\lambda_{n}=\log(t_{n})$. Note that in the case of $Z_{r}(s)$, the change of variables is actually $\lambda_{k}=2\log(t_{k})$, since the $t_{k}$ is squared there. Given this, we prove the lemma. 

\begin{lemma}
	Let $r\in C(\R/2\pi\Z,(0,\infty))$. Then $Z_{r}(s)$ converges in the half plane $\Re(s)>1$.  
\end{lemma}
\tab Note throughout this proof, we use the notation of the first section. 
\begin{proof}
	From \cite{Hardy1}, $Z_{r}(s)$ converges for $\Re(s)>\gm$, where 
	$$\limsup\frac{\log(A(k))}{2\log(t_{k})}=\gm,$$
	where $A(k)=\sum_{n=1}^{k}a_{k}$. As noted before, $A_{k}$ is in fact the number of lattice points in the region $t_{k}D$. Since $r$ is continuous, there is a dilation of the unit square $Q_{1}$ with side length $M\in \Z$ which bounds the region $D$. Furthermore, there is a square $Q_{2}$ entirely contained in $D$ with side length $m\in \Z$. The number of lattice points in $t_{k}D$ is bounded above by $\lceil t_{k}Q_{1} \rceil$ and bounded below by $\lfloor t_{k}Q_{2} \lfloor $, and the number of lattice points in these are $(t_{k}M)^{2}$ and $(t_{k}m)^{2}$, respectively. Now,
	$$\lim\limits_{k\to\infty}\frac{2\log( \lfloor t_{k}m \rfloor )}{2\log(t_{k})}=1\leq 
	\lim\limits_{k\to \infty}{\log(A(k))}{2\log(t_{k})} \leq \lim\limits_{k\to\infty}\frac{2\log( \lceil t_{k}M \rceil )}{2\log(t_{k})}=1$$
	since $t_{k}$ increases to infinity as $k\to \infty$. Therefore $\gm=1$. 
\end{proof}

This immediately gives the following proposition

\begin{proposition}
	Let $r\in C(\R/2\pi\Z,(0,\infty))$. Then $Z_{r}(s)$ has a pole at $s=1$. 
\end{proposition}

\begin{proof}
	This is clear from Landau's theorem, which states that if a generalized Dirichlet Series has positive coefficients, then there is a singularity at the real point on the abcissa of convergence. Again, see \cite{Hardy1}.  
\end{proof}

We need another small lemma relating to the poles of the gamma factors which will be of use later. 

\begin{lemma}
	Let $S=\{s|s=\frac{k-k_{0}}{\alpha_{i}}+1,k\geq 0\}$ for some $k_{0}\in \Z_{\geq 0}$, $\beta_{i}\in\R_{>0}$. If $\frac{1}{2}+\ell\in S$ for some $\ell\in\Z$, then $S$ contains infinitely many integers in the domain $\Re(s)>\frac{1}{2}+\ell$. 
\end{lemma}
\begin{proof}
	Since $\frac{1}{2}+\ell\in S$, we may conclude that for some integer $k_{1}\in \Z_{\geq 0}$, $\beta_{i}=\frac{k_{1}-k_{0}}{\frac{1}{2}+\ell}$. Then the elements of $S$ are of the form $s=\frac{k-k_{0}}{k_{1}-k_{0}}(\ell+\frac{1}{2})+1$. Then $s$ is an integer greater than $\frac{1}{2}+\ell$ as long as $\frac{k-k_{0}}{k_{1}-k_{0}}$ is an even positive integer. As $k$ ranges through positive integers, this happens infinitely many times.
\end{proof}

We are now in position to prove the theorem.

\begin{proof}
	\tab Suppose there is $\alpha_{i},\beta_{j},\mu_{i},\om_{j},A,B$ as in the definition as above:
	$$8\zeta(2s-1)=\frac{A^{1-s}\prod \Gamma(\alpha_{i}(1-s)+\mu_{i})}{B^{s}\prod \Gamma(\beta_{j}s+\om_{j})}Z(r^{\ast},1-s)$$
	for some continuous $r^{\ast}$.
	Now, this should be the same as for $s\mapsto 1-s$ so that 
	$$8\zeta(1-2s)=\frac{A^{s}\prod \Gamma(\alpha_{i}s+\mu_{i})}{B^{1-s}\prod \Gamma(\beta_{j}(1-s))+\om_{j})}Z(r^{\ast},s).$$
	Multiplying both sides by $\pi^{s-1/2}\Gamma(1/2-s)$, we obtain
	$$8\pi^{s-1/2}\Gamma(1/2-s)\zeta(1-2s)=8\pi^{-s}\Gamma(s)\zeta(2s)=\frac{\pi^{s-1/2}A^{s}\Gamma(1/2-s)\prod \Gamma(\alpha_{i}s+\mu_{i})}{B^{1-s}\prod \Gamma(\beta_{j}(1-s)+\om_{j})}Z(r^{\ast},s)$$
	Which means, if we multiply through, that 
	$$8\pi^{-s}B^{1-s}\Gamma(s)\prod \Gamma(\beta_{j}(1-s)+\om_{j})\zeta(2s)=\pi^{s-1/2}A^{s}\Gamma(1/2-s)\prod \Gamma(\alpha_{i}s+\mu_{i})Z(r^{\ast},s)$$
	From here, we analyze the poles and holomorphicity of the two sides of this equation in the region $\Re(s)\geq 1$.\\
	\tab Let $P_{L}$ and $P_{R}$ be the set of poles of the left and right hand sides in $D=\{s|\Re(s)\geq 1\}$. These two sets are assumed equal. The poles of the left hand side are, explicitly,
	$$P_{L}=D\cap\left( \bigcup \left\{s|s=\frac{k+\om_{j}}{\beta_{j}}+1,k\in \Z_{\geq 0}\right\}\right)=D\cap \left(\bigcup S_{j}\right)$$
	where $S_{j}$ is the set of poles of $\Gamma(\beta_{j}(1-s)+\om_{j})$. The set $P_{R}$ is more delicate.
	The $\Gamma(\frac{1}{2}-s)$ contributes poles at $\{s|s=k+\frac{1}{2},k\in \Z_{\geq 0}\}$. In this region, the product of gamma factors $\Gamma(\alpha_{i}s+\mu_{i})$ contributes only finitely many poles and no zeros. From \cite{Hardy1}, we know that $Z(r^{\ast},s)$ has finitely many zeros in this region, and the lemma states that $Z(r^{\ast},s)$ has a pole at $s=1$. Since the LHS has no zeros here, we conclude that the only zeros of $Z(r^{\ast},s)$ may only cancel out the poles of the other factors on the RHS. Thus, 
	$$P_{R}=D\cap \left(\{1\}\cup \{s|s=\frac{1}{2}+k,k\in\Z,k\geq N\geq 0 \} \right)\cup F,$$
	where $F$ is the finite set of poles coming from the gamma factors and not canceled by $Z(r^{\ast},s)$, and $N$ depends on the zeros of $Z(r^{\ast},s)$. Note that for some $c$, $\{s|\Re(s)>c\}\cap P_{R}$ contains infinitely many half integers. Furthermore, $P_{R}$ contains only finitely many integers.  \\
	\tab Now, since $P_{L}=P_{R}$, $1\in P_{L}$. This means that one of the $S_{j}$, which we will assume without loss of generality is $S_{1}$, is the set 
	$$S_{1}=\left\{s|s=\frac{k-k_{0}}{\beta_{1}}+1,k\in \Z_{\geq 0}    \right\}$$
	for some $k_{0}\in \Z_{\geq 0}$. The set $S_{1}\subset P_{R}$ as well. Now, since $S_{1}$ contains elements of arbitrarily large real part, $S_{1}$ must also contain a half integer. By the lemma, then $S_{1}$ must contain infinitely many integers as well. This contradicts the fact that $P_{R}$ contains only finitely many integers. 
\end{proof}

Notably, this proof relied on the pole of $Z_{r}(s)$ at $s=1$, and this relies heavily on analytic properties of $r$. The pole there made the set of poles "uneven" in a sense. 

\subsection*{An Odd Shape}

We describe a shape which has the same Hlawka Zeta Function as the square. Its boundary is the union of the following lines: 
$$L_{1}=\{(x,y)\in \R^{2} | -1\leq x \leq 1, y=-1\},$$
$$L_{2}=\{(x,y)\in \R^{2} | -1\leq y \leq 1, x=-1\},$$
$$L_{3}=\{(x,y)\in \R^{2}|-1\leq x \leq 0, y=\frac{1}{2}-\frac{1}{2}x\},$$
$$L_{4}=\{(x,y)\in \R^{2}|0\leq x \leq 1, y=\frac{1}{2}+\frac{1}{2}x\},$$
$$L_{5}=\{(x,y)\in \R^{2} | 1\leq x \leq 2, y=1\},$$
$$L_{6}=\{(x,y)\in \R^{2} | 1\leq x \leq 0, y=x-1\},$$
$$L_{7}=\{(x,y)\in\R^{2} | -1 \leq y \leq 0, x=1\}.$$ 
This shape, call it $O$, has the same area as the square centered at the origin with side length 2. One can show the following about this shape fairly simply: 
\begin{proposition}
	The shape $O$ has $t_{k}=k$ and $a_{k}=8k$. Thus, if $Q$ is the square, then $Z(O,s)=Z(Q,s)$. 
\end{proposition}

\begin{proof}
	It is clear that $t_{k}=k$, and it is also clear that this region is star-shaped as well. Now, the idea for the $a_{k}$ is that the vertical and horizontal lines account for the amount of lattice points as the square does, so it suffices to show that the "corner" and "bump" account for the same amount of lattice points. This is readily checked. 
\end{proof}

This sort of counterexample is fairly degenerate, but shows that if there is a conjecture to be made about the fibers of $r\mapsto Z_{r}(s)$, perhaps there need to be a stronger condition on the class of functions considered.	
	
\section{Final Remarks}

\subsection*{Further Study}
There are a number of open questions still about the Hlawka Zeta Function, but these results should lend some light on the situation. One might be tempted to pursue the following: \\
\tab (1): The methodology that was used to find interpret the functional equation for the ellipse may be able to find any sort of general functional equation. Namely, one can use the Fourier series of $r(\theta)$ in terms of $\cos^{n}(\theta)$ and $\sin^{n}(\theta)$, then use the extended binomial theorem and basis change to find the $4q$ Fourier coefficients of $r(\theta)^{s}$. This will involve even more, far more difficult combinatorics. Then, if $Z(r,s)$ has a functional equation $s\mapsto 1-s$, the functional equation of the $(E^{\sharp}.f_{q})(I,s)$ should say something about how the Fourier coefficients change. \\
\tab (2): While we have found an example where $Z_{r}(s)=Z_{r^{\ast}}(s)$ where $r^{\ast}$ is not in the $GL(2,\R)$-orbit of $r$, this does not disprove our conjecture about the stabilizer of the induced action of $GL(2,\R)$ on $Z_{r}(s)$, $g\cdot Z_{r}(s)=Z_{g\cdot r}(s)$. This conjecture is still open. Furthermore, we have reason to believe that the following conjecture may be true: 
\begin{conjecture}
	Let $r\in C^{1}(\R/2\pi\Z,(0,\infty))$, i.e. $r$ is continually differentiable. Then $Z_{r}(s)=Z_{r^{\ast}}(s)$ implies that $r^{\ast}$ is in the $GL(2,\R)$ orbit of $r$. 
\end{conjecture}
This conjecture has basis in our examination of the asymptotics of $t_{k}$ and $a_{k}$.

\subsection*{Acknowledgments}

\tab This research was under the tutelage of Joseph Hundley at SUNY Buffalo, who has pushed me ever-forward throughout the year long-process. I am grateful for the opportunity to work under and along with him through this time. The original definitions and motivation for this work is found in \cite{Huxley1}. Furthermore, this research was supported by the Summer Math Scholarship over the span of 2017. Without this support, the extent of results would not have been possible.

\bibliography{math}

\bibliographystyle{amsra}

\end{document}